\documentclass{amsart}

\usepackage[utf8]{inputenc}
\usepackage[T1]{fontenc} 
\usepackage{amsmath}
\usepackage{amsthm} 
\usepackage{amscd} 
\usepackage{amssymb}
\usepackage{latexsym} 

\usepackage[english]{babel}

\usepackage{hyperref}

\usepackage{multirow}

\usepackage[english]{babel}

\usepackage{graphicx} 

\usepackage{epsfig}

\newcommand{\cal}{\mathcal}

\def\epsilon{\varepsilon}

\renewcommand{\phi}{\varphi}

\DeclareMathOperator{\ind}{\mathrm{ind}}

\newcommand{\Out}{\text{Out}} 
\newcommand{\Aut}{\text{Aut}}

\newcommand{\FN}{{F}_N} 
\newcommand{\cvn}{\text{cv}_N}

\newcommand{\cvnbar}{\bar{\text{cv}}_N}

\newcommand{\R}{\mathbb R} 
\newcommand{\Z}{\mathbb Z}
 
\newcommand{\N}{\mathbb N}

\def\strutdepth{\dp\strutbox}
\def \ss{\strut\vadjust{\kern-\strutdepth \sss}}
\def \sss{\vtop to \strutdepth{
\baselineskip\strutdepth\vss\llap{$\diamondsuit\;\;$}\null}}

\def\strutdepth{\dp\strutbox}
\def \sst{\strut\vadjust{\kern-\strutdepth \ssss}}
\def \ssss{\vtop to \strutdepth{
\baselineskip\strutdepth\vss\llap{$\spadesuit\;\;$}\null}}

\def\strutdepth{\dp\strutbox}
\def \ssh{\strut\vadjust{\kern-\strutdepth \sssh}}
\def \sssh{\vtop to \strutdepth{
\baselineskip\strutdepth\vss\llap{$\heartsuit\;\;$}\null}}

\def\bar{\overline} 
\def\tilde{\widetilde} 

\newtheorem{thm}{Theorem}[section]
\newtheorem{cor}[thm]{Corollary} 
\newtheorem{lem}[thm]{Lemma}
\newtheorem{prop}[thm]{Proposition}

\newtheorem*{thm*}{Theorem}
\newtheorem*{prop*}{Proposition}
\newtheorem*{thm-main*}{Theorem~\ref{thm:main}}

\theoremstyle{definition} 

\newtheorem*{defn*}{Definition} 
\newtheorem{rem}[thm]{Remark}
\newtheorem*{rem*}{Remark}
\newtheorem{convention}[thm]{Convention}

\theoremstyle{remark}

\numberwithin{equation}{section}

\begin{document}

\title{Index realization for automorphisms of free groups}

\author{Thierry Coulbois and Martin Lustig}

\address{Institut de mathématiques de Marseille\\ 
Université d'Aix-Marseille\\
39, rue Frédéric Joliot Curie \\
13453 Marseille Cedex 13\\
France\\
\href{mailto:thierry.coulbois@univ-amu.fr}{\nolinkurl{thierry.coulbois@univ-amu.fr}}\\
\href{mailto:martin.lustig@univ-amu.fr}{\nolinkurl{martin.lustig@univ-amu.fr}}
}




\begin{abstract} 
For any surface $\Sigma$ of genus $g \geq 1$ and 
(essentially) any collection of
positive integers $i_1, i_2, \ldots, i_\ell$ with $i_1+\cdots +i_\ell
= 4g-4$ Masur and Smillie~\cite{MS} have shown that there exists a
pseudo-Anosov homeomorphism $h:\Sigma \to \Sigma$ with precisely
$\ell$ singularities $S_1, \ldots, S_\ell$ in its stable foliation
$\cal L$, such that $\cal L$ has precisely $i_k+2$ separatrices raying
out from each $S_k$.

In this paper we prove the analogue of this result for automorphisms
of a free group $\FN$, where ``pseudo-Anosov homeomorphism'' is
replaced by ``fully irreducible automorphism'' and the Gauss-Bonnet
equality $i_1+\cdots +i_\ell = 4g-4$ is replaced by the index
inequality $i_1+\cdots +i_\ell \leq 2N-2$ from \cite{GJLL}.
\end{abstract}

\maketitle

\section{Introduction}

In \cite{GJLL} for every automorphism $\Phi \in \Aut(\FN)$ of a
non-abelian free group $\FN$ of finite rank $N \geq 2$ an {\em index}
$\ind (\Phi)$ has been defined, which counts in a natural way
attracting fixed points at the Gromov boundary $\partial \FN$ and the
rank of the fixed subgroup ${\rm Fix}(\Phi)$ of $\Phi$. If ${\rm
  Fix}(\Phi) = \{1\}$, then $2\ind( \Phi) + 2$ is simply the number of
attractors of the homeomorphism $\partial \Phi:\partial\FN \to
\partial \FN$ induced by $\Phi$.

As main result of \cite{GJLL} the index inequality 
\[
\ind (\Phi) \leq N-1
\]
has been proved, which strengthens the celebrated Scott conjecture,
proved in~\cite{BH}, and also extends some well known consequences of
Nielsen-Thurston theory for surface homeomorphisms to free group
automorphisms, in particular after passing to the {\em stable index}
$\ind_{stab} (\phi$) of the associated outer automorphism $\phi$,
defined below in (\ref{eq:index-inequality-star}) as sum of $\ind
(\Phi_k)$ for suitable representatives $\Phi_k$ of a positive power of
$\phi$.

The main difference to surface homemorphisms, however, where the
analog indices always sum up to give via Gauss-Bonnet the maximal
possible value postulated in (\ref{eq:index-inequality-star}), is that
$\ind_{stab} (\phi)$ may well be strictly
smaller than $N-1$.  Ever since it has been an open question which
precise value of $\{\frac{1}{2}, 1, \frac{3}{2}, \ldots,$
$\frac{2N-3}{2}, N-1\}$ can be realized as stable index $\ind_{stab}
(\phi)$ by 
some 
$\phi \in \Out(\FN)$, in particular if one restricts to
automorphisms $\phi$ of $\FN$ which are {\em irreducible with irreducible powers (iwip)}, also called {\em fully irreducible} (see
Section~\ref{sec:prelim-CLP}).

For any given $\phi \in \Out(\FN)$ its representatives $\Phi_k \in
\Aut(\FN)$ are partitioned into {\em isogredience classes}, where
isogredient automorphisms are conjugate by inner automorphisms and
hence have conjugate $\partial \FN$-dynamics and thus equal
indices. 
It follows 
from the results of \cite{GJLL} 
that any $\Phi_k$ has a positive power $\Phi_k^{m_k}$ for which (as well as for all of its powers) the fixed subgroup and the number of attracting fixed points on $\partial \FN$ is maximal; the index of $\Phi_k^{m_k}$ will be called the {\em stable index} of $\Phi_k$ and denoted by $\ind_{stab} \Phi_k$.

The {\em stable index list} of $\phi$ is defined as the
longest sequence (up to permutation) of positive indices $\ind_{stab}
(\Phi_1),\ind_{stab} (\Phi_2),\ldots,\ind_{stab} (\Phi_\ell)$, given by pairwise
non-isogredient representatives $\Phi_k$ of some power $\phi^t$, for
any fixed $t\geq 1$. The inequalities
\begin{equation}\label{eq:index-inequality-star}
\frac{1}{2} \,\, \leq \,\, \ind_{stab} \phi := \sum_{k=1}^\ell \ind_{stab}
\Phi_k \,\, \leq \,\, N-1
\end{equation}
have been shown in \cite{GJLL}.  Handel and Mosher~\cite[Question~6 in
  \S1.5]{HM-axes} have asked explicitly which such values are realized
as stable index list by iwip automorphisms of $\FN$.  We denote such a
(potential) index list by $[j_1, \ldots, j_\ell]$, where the $j_k$ are
usually given in decreasing order.

For the ``maximal'' case, i.e.  $\ind_{stab} (\phi) = N-1$, 
an almost 
complete
answer to this question has been given by Masur and Smillie~\cite{MS}:
For $N \geq 3$ any list $[j_1, j_2, \ldots, j_\ell]$ of positive $j_k
\in \frac{1}{2} \Z$, with $\sum j_k = N-1$ (other than the single
exceptional case $[\frac 32,\frac 32]$ for $N = 4$, see
Section~\ref{sec:proof-discussion} below), can be realized as index
list of an iwip automorphism which is {\em geometric}, i.e. $\phi$ is
induced by a pseudo-Anosov homeomorphism of a surface with one
boundary component. On the other hand, if $\ind_{stab} (\phi) \leq N -
\frac{3}{2}$, then any iwip $\phi$ is known not to be geometric, 
and in particular for any representative $\Phi_k \in
\Aut(\FN)$ the fixed subgroup is trivial.
The
purpose of this paper is to show that the analogue of Masur and
Smillie's result holds also in the non-maximal case:

\begin{thm}\label{thm:main} 
Let $N \geq 3$, and let $j_1,j_2,\ldots,j_\ell$ be any list of
positive numbers from $\frac{1}{2} \Z$ which satisfy:
\[\frac{1}{2} \,\, \leq \,\,\sum_{k=1}^{\ell} j_k \,\,\leq \,\,
N-\frac{3}{2}\] Then there exists (and we give an explicit
construction) an iwip automorphism $\phi \in \Out(\FN)$ which realizes
the given list of values $j_k$ as stable index list.
\end{thm}

For $N=3$ the statement of the theorem had already been proved 
by 
C.~Pfaff \cite{Pfaff}. Other special cases were also known, for
example the single element list 
$[N-\frac{3}{2}]$ 
for any $N \geq 3$
(see~\cite{JL}). A further discussion, including some experimental
data obtained by the first author, is given in
section~\ref{sec:proof-discussion} below (compare also~\cite[Section
  VI]{GJLL}).
  
\begin{rem}
\label{branchin-index++}
From Theorem \ref{thm:main} one deduces directly as corollary an analogous existence statement for indecomposable $\R$-trees $T$ with free isometric $\FN$-action that have prescribed branching index list given by the numbers $j_k$. This follows directly from the material assembled in Section 8 of \cite{CL}.  The authors do not know whether such an existence statement was known previously.
\end{rem}

\medskip

Already in~\cite{GJLL} the relationship between the index of $\phi$
and the branching index of a forward limit $\R$-tree tree $T$ of
$\phi$ has been exploited (compare also~\cite{HM-axes}). If $\phi$ is
iwip, then such $T$ in the Thurston boundary $\partial \cvn$ of
(unprojectivized) Outer space $\cvn$ is unique up to rescaling, and
for non-geometric $\phi$ the isometric $\FN$-action on $T$ is free and
has dense orbits. For a suitable exponent $t\geq 1$ there is a natural
1-1 correspondence, between isogredience classes of representatives
$\Phi_k$ of $\phi^t$ with $\ind_{} (\Phi_k) >0$ on one hand, and
$\FN$-orbits of branch points $P_k$ of $T$ on the other, where $2
\ind_{stab} (\Phi_k) + 2$ is precisely equal to the number of
directions at $P_k$. An exposition of this relationship is given in
\S8 of \cite{CL}.

This correspondence can be carried one step further by using the fact
that $T$ is obtained as projective limit (in $\cvnbar = \cvn \cup
\partial \cvn$) of simplicial metric trees $\tilde \Gamma$ with free
isometric $\FN$-action, which occur naturally as universal cover of a
train track representative $f:\Gamma \to \Gamma$ of $\phi$ (see
\S\ref{sec:prelim-CLP}). Such train track representatives carry an
{\em intrinsic gate structure} which allows one to define a gate index
at every vertex of $\Gamma$ and a gate index list by considering all
periodic vertices of $\Gamma$ with 3 or more gates. There is a natural
relationship between the gates of $\tilde \Gamma$ and the branching
directions of $T$, which in the absence of so called {\em periodic
  INPs} (see \S\ref{sec:prelim-CLP} below) becomes a 1-1
correspondence. Again, see \S8 of \cite{CL} for more details.

The problem of realizing a given list $[j_1,j_2,\ldots, j_\ell]$ as in
Theorem~\ref{thm:main} as stable index list of an iwip automorphism
can hence be subdivided into the following subproblems:
\begin{enumerate}
\item\label{problem:graph} Construct a graph $\Gamma$ with vertices
  $v_1, \ldots, v_\ell$ and define a gate structure $\mathbf G$ on
  $\Gamma$ which realizes the given list of the values $j_k$ as gate
  indices at the vertices $v_k$.
\item\label{problem:iwip-but-inp} Define a map $h: \Gamma \to \Gamma$
  which respects the gate structure $\mathbf G$ and is ``iwip up to
  INPs''.
\item\label{problem:control-inp} Control the periodic INPs of $h$.
\end{enumerate}
Subproblems~(\ref{problem:graph}) and (\ref{problem:iwip-but-inp}) are
solved below in sections~\ref{sec:graph} and
\ref{sec:map-iwip-but-inp}. Subproblem~(\ref{problem:control-inp}),
which is the hardest and conceptually the most interesting, requires a
new tool, called {\em long turns}, which has been provided and
investigated by the first two authors in the ``companion
paper''~\cite{CL}.  In section~\ref{sec:legalizing-map} we give a
brief summary of this method and provide the concrete tools that allow
us in section~\ref{proofs} to apply the results of \cite{CL} in order
to obtain a {\em legalizing} train track morphism $g: \Gamma \to
\Gamma$. It is then shown how Theorem~1.1 of \cite{CL} (quoted in
section~\ref{proofs} in an appropriate version) can be applied to
solve the left-over Subproblem~(\ref{problem:control-inp}) for the
resulting train track map $h\circ g$.

\smallskip

\noindent{\em Acknowledgements:} This paper was intended by the
authors to be joint with Catherine Pfaff: a large part of it is rooted
in our weekly discussions with Catherine, during the months before she
left Marseille. We regret that she declined despite our insistence to
be coauthor of the paper.

We also would like to point the reader to the thesis work of Sonya
Leibman~\cite{leibman}, which came only very recently to our
attention. Some of her results seem to be very interesting to the
context of the work presented here; in particular, there is an overlap
of the results of her section 5.2 (Lemma 5.4) and our subsection 7.1
below.

\section{Preliminaries}
\label{sec:prelim-CLP}

We will use in this paper the same terminology as set up in sections~2
and 3 of \cite{CL}:

A {\em graph} $\Gamma$ is always connected, without vertices of
valence $1$ or $2$, and moreover it is finite, unless it is the
universal covering of a finite connected graph.  The {\em edges}
$E^\pm(\Gamma)$ of $\Gamma$ come in pairs $e, \bar e$ which differ
only in their orientation, and $E^+(\Gamma)$ contains precisely one of
the two elements in each pair.

A {\em gate structure} $\mathbf G$ on $\Gamma$ is a partition of the
edges $e \in E^\pm(\Gamma)$ into equivalence classes $\mathfrak g_i$
(called {\em gates}), where equivalent edges must have the same
initial vertex $v$.  Two edges $e, e' \in E^\pm(\Gamma)$ with same
initial vertex form a {\em turn} $(e, e')$, which is called {\em
  legal} (with respect to $\mathbf G$) if $e$ and $e'$ belong to
distinct gates, and {\em illegal} if they belong to the same gate. The
turn $(e, e')$ is called {\em degenerate} if $e = e'$.

A path $\gamma = e_1 e_2 \ldots e_q$ {\em crosses over} a {\em gate
  turn} $(\mathfrak g_i, \mathfrak g_j)$ if for some $k \in \{1,
\ldots, q-1\}$ one has $\bar e_k \in \mathfrak g_i, e_{k+1} \in
\mathfrak g_j$ or $\bar e_k \in \mathfrak g_j, e_{k+1} \in \mathfrak
g_i$.  The path $\gamma$ is {\em legal} if, for each $k \in \{1,
\ldots, q-1\}$, the edges $\bar e_k$ and $e_{k+1}$ belong to different
gates of $\mathbf G$ (i.e. $\gamma$ crosses only over legal turns).

The {\em gate index} $\ind_{\mathbf G}(v)$ at a vertex $v$ is given by
$\ind_{\mathbf G}(v) := \frac{g(v)}{2} - 1$, where $g(v)$ denotes the
number of gates at $v$.

A {\em graph map} $f: \Gamma \to \Gamma$ maps vertices to vertices and
edges to (possibly unreduced) edge paths. The map $f$ has {\em no
  contracted edges} if for any edge $e$ of $\Gamma$ the combinatorial
length (= number of edges traversed) of $f(e)$ satisfies $|f(e)|\geq 1$. In this
case $f$ induces a well defined map $Df$ on $E^\pm(\Gamma)$ which maps
the edge $e$ to the initial edge of the path $f(e)$.

The {\em transition matrix} $M(f) = (m_{e', e})_{e', e \in
  E^+(\Gamma)}$ of $f$ is defined as non-negative matrix, where
$m_{e', e}$ counts the number of times that $f(e)$ crosses over $e'$
or over $\bar e'$. The equality
\[
M(f\circ g) = M(f)M(g)
\] 
is a direct consequence of the definition of the transition matrix.
Recall that a non-negative matrix $M$ is called {\em primitive} if
some positive power $M^t$ is {\em positive}, i.e. all coefficients of
$M^t$ are strictly positive.

A graph map $f: \Gamma \to \Gamma$ is a {\em train track morphism},
with respect to a given gate structure $\mathbf G$ on $\Gamma$, if it
has no contracted edges, and if $f$ maps every legal path to a legal
path.  It is shown in~\cite{CL} that a train track morphism has the
additional property that at every periodic vertex $v$ of $\Gamma$ any
illegal turn is mapped to an illegal turn, or equivalently: $f$
induces at every periodic vertex $v$ of $\Gamma$ a bijective map from
the gates at $v$ to the gates at $f(v)$. Note that in this paper all
train track morphisms that occur have only periodic vertices; indeed,
each vertex is a fixed point.

For a graph $\Gamma$ without preassigned gate structure, a {\em train
  track map} $f:\Gamma\to\Gamma$ in the classical sense as defined by
Bestvina and Handel~\cite{BH} (and hence in particular any {\em train
  track representative} of a given automorphism of $\FN$) is a graph
map with no contracted edges with the property that for any $t>0$ and
any edge $e$, $f^t(e)$ is a reduced path.

As legal paths are reduced, any train track morphism $f:\Gamma \to
\Gamma$ is always a classical train track map. Conversely, every
classical train track map $f:\Gamma \to \Gamma$ is a train track
morphism with respect to the {\em intrinsic} gate structure $\mathbf
G(f)$ on $\Gamma$, defined by $f$ through declaring two edges $e, e'$
with same initial vertex to lie in the same gate if and only if for
some $t \geq 1$ the edge paths $f^t(e)$ and $f^t(e')$ have non-trivial
initial subpaths in common.  Notice however that, for a train track
morphism $f$ with respect to some gate structure $\mathbf G$, the
intrinsic gate structure $\mathbf G(f)$ may be strictly finer than the
given gate structure $\mathbf G$.

A reduced path $\gamma \circ \gamma'$ in $\Gamma$ is a {\em periodic
  INP} for a train track morphism $f: \Gamma \to \Gamma$ if $\gamma$
and $\gamma'$ are legal and for some $t \geq 1$ the path $f^t(\gamma
\circ \gamma')$ is homotopic relative endpoints to $\gamma \circ
\gamma'$.

The {\em gate-Whitehead-graph} $Wh_{\mathbf G}^v(f)$ of a train track
morphism $f: \Gamma \to \Gamma$ at a vertex $v$ of $\Gamma$ has the
gates $\mathfrak g_i$ of $\mathbf G$ at $v$ as vertices and a
(non-oriented) edge connecting $\mathfrak g_i$ to $\mathfrak g_j$ if
and only for some edge $e$ of $\Gamma$ and some integer $t \geq 1$ the
edge path $f^t(e)$ crosses over the gate turn $(\mathfrak g_i,
\mathfrak g_j)$.

Recall also that an automorphism $\phi \in \Out(\FN)$ is called {\em
  iwip} (or {\em fully irreducible}) if no positive power of $\phi$
fixes the conjugacy class of any proper free factor of $\FN$.

\section{The graph $\Gamma$ with prescribed index list}\label{sec:graph}

In this and the following sections, let $N, \ell$ and $j_1,\ldots,
j_\ell$ be given as in Theorem~\ref{thm:main}.  In this section we
will build a graph $\Gamma$ with $\pi_1\Gamma \cong \FN$ which has
precisely $\ell$ vertices $v_1, \ldots, v_\ell$, and has at each
vertex $v_k$ precisely $i_k := 2j_k + 2 \geq 3$ gates: the gate
structure $\mathbf G$ on $\Gamma$ realizes the given list
$j_1,\ldots,j_\ell$ as gate index list.

Note that from the inequalities in Theorem~\ref{thm:main} we obtain 
\[1 \leq i_1+\cdots+i_\ell - 2 \ell\leq 2N-3\]
as initial assumption on the number of gates in $\Gamma$.

We divide the possible index lists in three different cases: 
\begin{enumerate}
\item The {\em even case}: $i_1+\cdots+i_\ell$ is even (that is to say $j_1+\cdots+j_\ell$ is an integer 
$\leq N-2$).
\item The {\em odd case} (non-maximal): $i_1+\cdots+i_\ell$ is odd and smaller than $2N-4 + 2 \ell$ 
(alternatively:  $j_1+\cdots+j_\ell \leq N-2$).
\item The {\em maximal odd case}: $i_1+\cdots+i_\ell=2N-3 + 2 \ell$ 
(i.e. $j_1+\cdots+j_\ell=N-\frac 32$).
\end{enumerate} 

We consider a circle which is subdivided at vertices labeled
$v_1,\ldots, v_\ell$, to obtain oriented edges labeled
$c_1,\ldots,c_\ell$ such that $c_k$ starts at $v_k$ and ends at
$v_{k+1}$ (for $k$ understood modulo $\ell$).  Note that if $\ell = 1$
then $c_1 = c_\ell$ is a loop edge at the sole vertex $v_1$ of
$\Gamma$.

At each vertex $v_k$ we add $i_k-2$ germs of edges to this circle. In
the odd and maximal odd cases we remove one of these germs at $v_1$
such that in any cases the number of germs is even. We group these
germs arbitrarily into pairs to form $r$ oriented edges
$b_1,\ldots,b_r$.  Here $r$ is the largest integer $\leq
j_1+\cdots+j_\ell$, with $r=0$ exactly if the index list is equal to
$[\frac 12]$.

In the even and odd cases let $s=N-r-1$, where we note that $s\geq 1$.
We add $s$ oriented edges $a_1,\ldots,a_s$ which are loops at the
vertex $v_1$.

In the maximal odd case we set $s=1$ and add a single edge $a_1$ which
is a loop at $v_1$.

Finally, in the odd case we add an extra edge $d$ which is a loop at
$v_1$.

\begin{figure}[h]
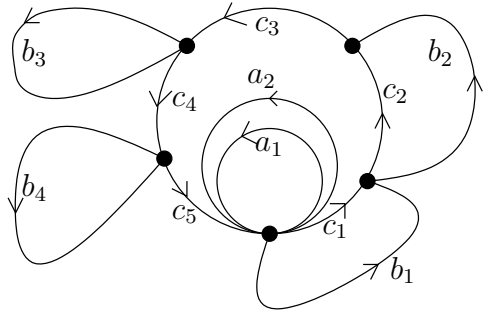 
\caption{\label{fig:graph-even} Even case: a graph $\Gamma$ with index list $[\frac 12,1,\frac 12,1,1]$ and $N=7$}
\end{figure}

\begin{figure}[h]
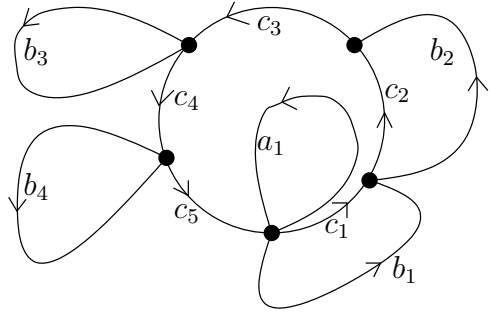 
\caption{\label{fig:graph-maximal} Maximal odd case: a graph $\Gamma$
  with index list $[1,1,\frac 12,1,1]$ and $N=6$}
\end{figure}

\begin{figure}[h]
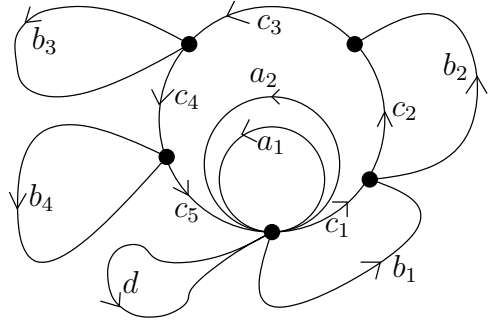 
\caption{\label{fig:graph-odd} Odd case: a graph $\Gamma$ with index
  list $[1,1,\frac 12,1,1]$ and $N=8$}
\end{figure}

The graphs $\Gamma$ defined above are connected, without vertices of
valence $1$ or $2$, and have fundamental group $\FN$. The {\em
  oriented edge set} is given by $E^+(\Gamma) = \{c_1, \ldots, c_\ell,
b_1, \ldots, b_r, a_1, \ldots a_s\}$ in the even and maximal odd
cases, and by $E^+(\Gamma) = \{c_1, \ldots, c_\ell, b_1, \ldots, b_r,
a_1, \ldots a_s,d\}$ in the odd case. In all cases we have $\ell,s\geq
1$ and $r \geq 0$, with $r = 0$ if and only if we are in the odd case
with index list $[\frac 12]$.

We define the gate structure $\mathbf G$ on $\Gamma$ in such a way
that each gate consists of a single edge, except for the following
gates, all based at $v_1$:
\begin{itemize}
\item in the even case: $\mathfrak g_1=\{c_1,a_1,\ldots a_s\}$,
  $\mathfrak g_2=\{\bar{c}_\ell,\bar{a}_1,\ldots,\bar{a}_s\}$;
\item in the odd case: $\mathfrak g_1=\{c_1,a_1,\ldots, a_s\}$,
  $\mathfrak g_2=\{\bar{c}_\ell,\bar{a}_1,\ldots,\bar{a}_s\}$ and
  $\mathfrak g_3=\{d,\bar{d}\}$;
\item in the maximal odd case: $\mathfrak g_1=\{c_1,a_1\}$.
\end{itemize}
Notice that in the maximal odd case $\bar a_1$ and $\bar c_\ell$
belong to distinct gates.

As a consequence, at every vertex $v_k$ there are precisely $i_k$
gates, so that we obtain:

\begin{prop}
\label{prop:graph-index-realization}
The gate structure $\mathbf G$ on $\Gamma$ realizes the given list of
values $j_1,\ldots, j_\ell$ as gate indices at the vertices $v_1,
\ldots, v_\ell$ of $\Gamma$.  \qed
\end{prop}

The following will be used crucially in the subsequent sections:

\begin{lem}\label{lem:hamiltonian-circuit}
Let $\Gamma$ be the graph equipped with the gate structure $\mathbf G$
as defined above.
\begin{enumerate}
\item\label{item:ueuprime} For each edge $e\neq a_1$ in $E^+(\Gamma)$
  there exists a legal loop $ ueu'$ in $\Gamma$ which starts in
  $\mathfrak g_1$ does not end in $\mathfrak g_1$, does not pass
  through $a_1, \bar a_1$ or $\bar e$ and passes exactly once through
  $e$ (we allow $u$ or $u'$ to be trivial).
\item\label{item:v} For each gate turn $t=({\mathfrak g},{\mathfrak
  g}')$, except for gate turns involving $\{\bar{a}_1\}$ in the
  maximal odd case, there exists a legal loop $v$ which starts in
  ${\mathfrak g}_1$, does not end in ${\mathfrak g}_1$, does not pass
  through $a_1$ or $\bar a_1$, and crosses over the gate turn
  $({\mathfrak g},{\mathfrak g}')$.
\item\label{item:alphabeta} In the even and odd cases, for any edge
  $e$ in $\mathfrak g_1$ there exist legal paths $\alpha$ and $\beta$
  that do not pass through any of the $a_i$ or through $e$ (and
  neither through their inverses).  Furthermore, $e\alpha$ is a legal
  path which ends in $\mathfrak g_2$, and $\beta$ is a legal loop
  based at $v_1$ that does not start in $\mathfrak g_1$ or $\mathfrak
  g_2$, and does not end in $\mathfrak g_1$. We allow $\alpha$ to be
  trivial.
\item\label{item:alphabetaprime} Symmetrically, in the even and odd
  cases, for each edge $\bar{e}$ in $\mathfrak g_2$ there exist legal
  paths $\alpha'$ and $\beta'$ that do not pass through any of the
  $a_i$ or through $e$ (nor through their inverses), such that
  $\alpha'e$ is legal and starts in $\mathfrak g_1$, while the legal
  loop $\beta'$ is based at $v_1$ but does not start in $\mathfrak
  g_2$ and does not end in $\mathfrak g_1$ or $\mathfrak g_2$. We
  allow $\alpha'$ to be trivial.
\end{enumerate}
\end{lem}

\begin{proof}
The above statements~(\ref{item:ueuprime}) and (\ref{item:v}) are easy
to verify if one keeps in mind that at every vertex there are $\geq 3$
gates, and that every vertex $v_k$ can be reached from $v_1$ by any
one of two disjoint paths on the circle $c_1\cdots c_\ell$, so that it
is easy to avoid any given edge in $\Gamma$.

Concerning statement~(\ref{item:alphabeta}), if $e\neq c_1$ or $\ell=1$, we let $\alpha$
be trivial. Otherwise we set $\alpha=c_2\cdots c_\ell$.  In the odd
case we let $\beta=d$. In the even case, there is at least one edge
$b_k$ (or $\bar b_k$) exiting from $v_1$. Let $v_{k'}$ be the endpoint
of $b_k$ (or of $\bar b_k$), and set $\beta=b_k$ if $k'=1$ and
$\beta=b_{k}c_{k'}\cdots c_\ell$ otherwise.

For statement~(\ref{item:alphabetaprime}) the paths $\alpha'$ and
$\beta'$ can be chosen symmetrically to $\alpha$ and $\beta$ in the
above case~(\ref{item:alphabeta}).
\end{proof}

\section{The train track morphism}
\label{sec:map-iwip-but-inp}

In this section we construct for the graph $\Gamma$ a train track
morphism $h: \Gamma \to \Gamma$ with respect to the gate structure
$\mathbf G$ specified in the last section.  The morphism $h$ will be
``fully irreducible up to INPs'' in that it has primitive transition
matrix $M(h)$ and connected gate-Whitehead-graphs at every vertex
(compare \cite[Propositions~4.1 and~4.2]{CL}).

Below we will consider graph maps $f: \Gamma\to\Gamma$ with the
following properties:

$(*)$\parbox{.8\textwidth}{\begin{enumerate}
\item\label{*:homot-eq} $f$ is a homotopy equivalence,
\item\label{*:train-track} $f$ is a
 train track morphism with respect to the gate structure $\mathbf G$, 
\item\label{*:fix-vertex} $f$ fixes each vertex of $\Gamma$,
\item\label{*:fix-gate} $f$ fixes each gate of $\mathbf G$, and
\item\label{*:e-image} the $f$-image of every edge $e$ crosses over $e$.
\end{enumerate}
}

\begin{lem}\label{lem:property-star}
Let $f_1$ and $f_2$ be graph maps which satisfy the above
Properties~$(*)$.
\begin{enumerate}
\renewcommand{\theenumi}{\alph{enumi}}
\item\label{item:composition-**} Then the composition $f_1 \circ
  f_2$ satisfies $(*)$ as well.
\item\label{item:composition-iwg} Moreover, for any vertex $v_k$
  of $\Gamma$ the gate-Whitehead-graph $Wh_{\mathbf G}^{v_k}(f_1 \circ
  f_2)$ contains both gate-Whitehead-graphs $Wh_{\mathbf
    G}^{v_k}(f_1)$ and $Wh_{\mathbf G}^{v_k}(f_2)$ as subgraphs.
\end{enumerate}
\end{lem}

\begin{proof}
Statement~(\ref{item:composition-**}) is a direct consequence of the
definition of properties
$(*)$. Statement~(\ref{item:composition-iwg}) has been proved under
slightly more general hypotheses as Proposition~3.10 in~\cite{CL}.
\end{proof}

The Properties~(\ref{*:fix-vertex}) and (\ref{*:fix-gate}) above imply
that a map which satisfies $(*)$ acts as identity on the set of gate
turns.  As a consequence one derives easily that
Statement~(\ref{item:composition-iwg}) of
Lemma~\ref{lem:property-star} can actually improved to $Wh_{\mathbf
  G}^{v_k}(f_1 \circ f_2) = Wh_{\mathbf G}^{v_k}(f_1) \cup Wh_{\mathbf
  G}^{v_k}(f_2)$.

\smallskip

We define below several graph maps on $\Gamma$ where we use the following:

\begin{convention} 
\label{convention:no-untouched-edges}
In this and the following sections, in the definition of a graph map
$\Gamma \to \Gamma$ we always use the convention that any edge with no
explicitly defined image is mapped identically to itself.
\end{convention}

For any edge $e\neq a_1 \in E^+(\Gamma)$ let $u$ and $u'$ be as in
Lemma~\ref{lem:hamiltonian-circuit}~(\ref{item:ueuprime}). Define
$h_e: \Gamma \to \Gamma$ by:
\[
h_e:\begin{array}[t]{rcl}a_1&\mapsto&ueu'a_1\\
e&\mapsto&eu'a_1ue
\end{array}
\]
Note that the $h_e$-image of $e$ passes through $a_1$ and that the
$h_e$ image of $a_1$ passes through $e$.

For any gate turn $t=({\mathfrak g},{\mathfrak g}')$ of $\Gamma$,
except for gate turns involving $\{\bar{a}_1\}$ in the maximal odd
case, let $v$ be as in Lemma~\ref{lem:hamiltonian-circuit}
(\ref{item:v}) and define $h_t:\Gamma\to\Gamma$ by:
\[
h_t: a_1\mapsto va_1
\]

Let $h'$ be the composition of all these maps $h_e$ (with $e\neq a_1$)
and $h_t$, where we do not care about the order of the composition.
Define $h:\Gamma\to\Gamma$ through $h=h'\circ h'$.

\begin{prop}\label{prop:h-positive} 
The map $h: \Gamma \to \Gamma$ is a train track morphism with respect
to $\mathbf G$.  Furthermore $h$ fixes every vertex $v_k$ of $\Gamma$,
maps each gate of the gate structure $\mathbf G$ to itself, and is a
homotopy equivalence.

In addition, the transition matrix $M(h)$ is positive, and the
gate-Whitehead-graph $Wh_{\mathbf G}^v(h)$ of $h$ at any vertex $v$ of
$\Gamma$ is connected.
\end{prop}

\begin{proof}
We first consider the maps $h_e$ with $e \neq a_1$ and $h_t$ as
defined above.  Properties~(\ref{*:train-track})-(\ref{*:e-image}) of
$(*)$ above are easily verified directly. For
Property~(\ref{*:homot-eq}) the reader can check directly that the map
given by
\[
\begin{array}[t]{rcl}
a_1&\mapsto&\bar u' \bar e \, \bar u a_1^2\\
e&\mapsto& \bar u \, \bar a_1 u e
\end{array}
\]
is a homotopy inverse of $h_e$.  The fact that it is not a train track
map with respect to $\mathbf G$ is irrelevant. For $h_t$ a homotopy
inverse is given simply by $a_1 \mapsto \bar v a_1$.  From
Lemma~\ref{lem:property-star} it follows now that the maps $h'$ and
$h$ have this Property~$(*)$, which is the statement of the first
paragraph to be shown.

In order to show that $M(h)$ is positive, we use the equality $M(fg) =
M(f) M(g)$ from Section~\ref{sec:prelim-CLP} and condition
(\ref{*:e-image}) of Property~$(*)$ to obtain that the image $h'(e)$
of any edge $e$ crosses over $a_1$, and that the image $h'(a_1)$ of
$a_1$ crosses over every edge $e\in E^+(\Gamma)$.  Hence the image
$h(e)$ of any edge $e$ passes through all edges of $\Gamma$. This
proves that the transition matrix $M(h)$ is positive.

From Lemma~\ref{lem:property-star} we know that the
gate-Whitehead-graph of $h$ contains that of $h_t$, for each gate turn
$t=({\mathfrak g},{\mathfrak g}')$.  It follows from the above
definition of $h_t$ via $a_1 \mapsto v a_1$ and from the definition of
$v$ in Lemma~\ref{lem:hamiltonian-circuit}~(\ref{item:v}) that in the
even and the odd cases the gate-Whitehead-graph of $h$ at each vertex
of $\Gamma$ is a complete graph and thus connected.  In the maximal
odd case there are no maps $h_t$ for the gate turns involving the gate
$\{\bar{a}_1\}$.  But in this case the gate turn
$(\{\bar{a}_1\},\{a_1, c_1\})$ is crossed over by $h_{c_1}(c_1)$.
This is enough to get that the gate-Whitehead-graph of $h$ at $v_1$ is
connected.
\end{proof}

\section{Building the legalizing map}
\label{sec:legalizing-map}

The goal of this (and the following) section is to construct a train
track morphism $g: \Gamma \to \Gamma$ with respect to $\mathbf G$
which is a homotopy equivalence and is ``legalizing''.  This notion
has been introduced in~\cite{CL}, and is now briefly summarized:

A pair $(\gamma, \gamma')$ of non-trivial legal paths $\gamma$ and
$\gamma'$ in $\Gamma$ is called in~\cite{CL} a {\em long turn} if the
{\em branches} $\gamma$ and $\gamma'$ start at the same vertex but
have distinct initial edges.  The set of long turns in $\Gamma$, with
both branches of length equal to some integer $C \geq 1$, is denoted
by $LT_C(\Gamma)$.

The long turn $(\gamma, \gamma')$ can be {\em legal} or {\em illegal},
according to whether its {\em starting turn} $s(\gamma, \gamma')$,
formed by the initial edges of $\gamma$ and $\gamma'$ respectively, is
legal or illegal (as defined for turns in the traditional sense, see
\S\ref{sec:prelim-CLP}).

If neither $g(\gamma)$ is a subpath of $g(\gamma')$ nor conversely,
then the long turn is called {\em $g$-long}, and the long turn,
obtained from $g(\gamma)$ and $g(\gamma')$ through erasing from both
the maximal common initial subpath, is called the {\em $g$-image} of
$(\gamma, \gamma')$ and denoted by $g^{LT}(\gamma, \gamma')$.  A train
track morphism $g$ is called {\em legalizing} if for some sufficiently
large constant $C \geq 1$ every long turn $(\gamma, \gamma') \in
LT_C(\Gamma)$ is $g$-long, and if $g^{LT}(\gamma, \gamma')$ is legal
(or, equivalenty, if the starting turn of $g^{LT}(\gamma, \gamma')$ is
legal).

To avoid a misunderstanding, we point out that any
non-degenerate turn $(e, e')$ in the classical sense can be considered
alternatively as long turn with branches of length 1.  In particular,
if $(e, e')$ is $g$-long, then it has both, an image turn $(Dg(e),
Dg(e'))$, as well as an image long turn $g^{LT}(e, e')$, which
furthermore has a starting turn $sg^{LT}(e, e')$ that is again a turn
in the classical sense.  However, in general $(Dg(e), Dg(e'))$ and
$sg^{LT}(e, e')$ will be quite different: for example, $(Dg(e),
Dg(e'))$ may well be degenerate, while $sg^{LT}(e, e')$ is by
definition always non-degenerate.

\smallskip

To construct the desired train track morphism $g$ we define now train
track morphisms $g_t$. In each of the cases considered below the
``variable'' $t$ denotes a non-degenerate illegal turn in $\Gamma$,
interpreted here as long turn with branch length 1.  The reader can
verify directly that all of the maps $g_t$ defined below satisfy the
statements~(\ref{*:train-track}), (\ref{*:fix-vertex}) and
(\ref{*:fix-gate}) of Property~$(*)$ from
Section~\ref{sec:map-iwip-but-inp}.  We use again
Convention~\ref{convention:no-untouched-edges}.

We first deal with the \underline{even and odd cases}:
\begin{enumerate}
\item Let $t=(a_i,e)$ be an illegal turn in $\Gamma$ with $1\leq i\leq
  s$, and with $e=c_1$ or $e=a_j$ where $i\neq j$.  Let $\alpha$ and
  $\beta$ be as in Lemma~\ref{lem:hamiltonian-circuit}. Define:
  \[
  g_t: \begin{array}[t]{rcl}a_i&\mapsto& a_i e\alpha\\ e&\mapsto&a_i
    \beta a_i e
    \end{array}
  \]
  The illegal turn $t=(a_i,e)$ is $g_t$-long and mapped by $g_t^{LT}$
  to the long turn $(e\alpha,\beta a_ie)$ which is legal.
\item Symmetrically, let $t=(\bar{a}_i,\bar{e})$ be an illegal turn in
  $\Gamma$ with $1\leq i\leq s$, and with $e=c_\ell$ or $e=a_j$ where
  $i\neq j$.  Let $\alpha'$ and $\beta'$ be as in
  Lemma~\ref{lem:hamiltonian-circuit}. Define:
  \[
  g_t: \begin{array}[t]{rcl}a_i&\mapsto& \alpha' ea_i\\ 
                             e&\mapsto&ea_i \beta' a_i
    \end{array}
  \]
  The illegal turn $t=(\bar{a}_i,\bar{e})$ is $g_t$-long and mapped by
  $g_t^{LT}$ to the long turn
  $(\bar{e}\,\bar{\alpha}',\bar{\beta}'\bar{a}_i\bar{e})$ which is
  legal.
\item In the \underline{odd case} we have one more illegal turn
  $t=(d,\bar{d})$. Define:
  \[
  g_t:\begin{array}[t]{rcl}
    a_1&\mapsto&a_1\bar d c_1\cdots c_\ell\\
    d&\mapsto&da_1\bar d
  \end{array}
  \]
The illegal turn $t=(d,\bar{d})$ is $g_t$-long and mapped by
$g_t^{LT}$ to the long turn $(a_1 \bar d,\bar{a}_1 \bar d)$ which is
legal.
\end{enumerate}

In the \underline{maximal odd case} there is only one illegal turn
$t=(a_1,c_1)$. As the rank $N$ is greater or equal to $3$, there is at
least one edge $b_1$ which starts from some vertex $v_{k}$ and ends at
some $v_{k'}$. We set $c_{[1,k]}=c_1\cdots c_\ell$ if $k=1$ and
$c_{[1,k]}=c_1\cdots c_{k-1}$ if $2 \leq k\leq\ell$.  We furthermore
set $c_{[k',\ell]}=1$ if $k'=1$, and $c_{[k',\ell]}=c_{k'}\cdots
c_\ell$ if $2 \leq k'\leq\ell$.

We define:
\[
  g_t: \begin{array}[t]{rcl} 
  a_1&\mapsto&c_1\cdots c_\ell c_{[1,k]}b_1c_{[k',\ell]}a_1\\ 
  c_1&\mapsto&c_1\cdots c_\ell c_{[1,k]}b_1c_{[k',\ell]}a_1c_1\\ 
  b_1&\mapsto&b_1c_{[k',\ell]}a_1c_{[1,k]}b_1
  \end{array}
\]

\begin{lem}
\label{lem:easy-check}
In the maximal odd case, every long turn of length equal to $\ell+1$
with starting turn $(a_1, c_1)$ is $g_t$-long and mapped by $g_t^{LT}$
to a legal long turn.
\end{lem}

\begin{proof}
Let $t^* = (a_1 e_2 \ldots e_{\ell +1}, c_1 e'_2 \ldots e'_{\ell +1})$
be the long turn under consideration. We first observe that, if $e_2
\neq c_1$
and $e_2 \neq a_1$\footnote{\,\, We'd like to thank C. Pfaff for having pointed out to us that the treatment of this case was missing in an earlier version of our paper.}, 
then $g_t^{LT}(t^*)$ has starting turn $(e_2, c_1)$ and
hence is legal. 
In order to treat computationally the possible ``exceptional'' cases without too many subcases we introduce a variable $x$ which we set to $x = a_1 c_1$ if $e_{2} = c_1$ and $x = a_1$ if $e_{2} = a_1$. 
Similarly, a second variable $y$ will be used below which is set to $y = a_1 c_1$ if $e'_{\ell +1} = c_1$ and $y = a_1$ if $e'_{\ell + 1} = a_1$.

We observe
that 
in each case 
$g_t^{LT}(t^*)$ is legal unless 
$e'_2 \ldots e'_{\ell +1} = c_2 \ldots c_\ell c_1$
or
$e'_2 \ldots e'_{\ell +1} = c_2 \ldots c_\ell a_1$.
We compute 
\[
g_t^{LT}(t^*) = (b_1c_{[k',\ell]}
x, 
c_k\cdots c_\ell c_{[1,k]}b_1c_{[k',\ell]}
y) \,\,\, {\rm if} \,\,\, 
\ell \geq k \geq 2 ,
\] 
and 
\[
g_t^{LT}(t^*) = (b_1c_{[k',\ell]}
x, 
c_1 \ldots c_\ell 
b_1c_{[k',\ell]}
y) \,\,\, {\rm if} \,\,\, 
k = 1 \,\,\, {\rm and} \,\,\,  \ell \geq 2.
\] 
Finally we have
\[
g_t^{LT}(t^*) = (b_1c_{[k',\ell]}
x, 
c_1 
b_1c_{[k',\ell]}
y) 
\,\,\, {\rm if} \,\,\, 
k = \ell = 1,
\] 
All three of those computed long turns are legal.
\end{proof}

We now verify:

\begin{lem}
\label{lem:homotopy-inverses}
Each of the above defined maps $g_t$ is a homotopy equivalence.
\end{lem}

\begin{proof}
For each case of the map $g_t$ we list below a map $g'_t$; the reader
can verify directly that they are homotopy inverses of the maps $g_t$.

\smallskip

Even and odd cases:

\noindent
(1)
 \[
 g'_t: \begin{array}[t]{rcl}
 a_i&\mapsto&   e \alpha \bar a_i \bar \beta \\ 
 e&\mapsto& \beta a_i \bar \alpha \, \bar e a_i \bar \alpha
    \end{array}
  \]

\noindent
(2)
 \[
 g'_t: \begin{array}[t]{rcl}
 a_i&\mapsto&   \bar \beta' \bar a_i \alpha' e \\ 
 e&\mapsto& \bar \alpha' a_i \bar e \, \bar \alpha' a_i \beta'
    \end{array}
  \]
 
 \noindent
(3)
 \[
 g'_t: \begin{array}[t]{rcl}
 a_1&\mapsto&  a_1 \bar c_\ell \ldots \bar c_1 d c_1 \ldots c_\ell \bar a_1 \\ 
 d&\mapsto& d c_1 \ldots c_\ell \bar a_1
    \end{array}
  \]

Maximal odd case:
\[
  g'_t: \begin{array}[t]{rcl} 
a_1&\mapsto& \bar c_{[k', \ell]} \bar b_1 \bar c_{[1, k]} a_1 \bar c_\ell \ldots \bar c_1 a_1^2 \bar c_\ell \ldots \bar c_1 a_1^2 \\ 
c_1&\mapsto&  \bar a_1 c_1 \\ 
b_1&\mapsto& \bar c_{[1, k]} \bar a_1 c_1 \ldots c_\ell \bar a_1 c_{[1, k ]} b_1
  \end{array}
\]
\end{proof}

We thus have proved:

\begin{prop}
\label{prop:local-legalizing}
  Let $L=1$ in the even and odd cases and $L=\ell+1$ in the maximal
  odd case, and let $t$ be any illegal turn of $\Gamma$. For each long
  turn $t^*$ of $\Gamma$, with branch length $L$ and with starting
  turn $t$, the map $g_t$ is a train track morphism with the property
  that $t^*$ is $g_t$-long and mapped by $g_t^{LT}$ to a legal long
  turn.
  
Furthermore, $g_t$ is a homotopy equivalence which fixes every vertex
of $\Gamma$ and every gate of $\mathbf G$.  \qed
\end{prop}

\section{Proof of the main result
}
\label{proofs}

Proposition~\ref{prop:local-legalizing} is the main ingredient we need
to obtain the desired legalizing map. This is done through an
application of Proposition~7.1 of~\cite{CL} which we quote now, for
the convenience of the reader in a slightly weakened form and with
terminology adapted to the present paper:

\begin{prop}[{\cite[Proposition~7.1]{CL}}]
\label{prop:quote-1}
Let $\Gamma$ be a graph with a gate structure $\mathbf G$. Assume that
there exists a constant $L\geq 1$, and assume furthermore:
\begin{enumerate}
\item\label{item:hyp-legalizing} For each illegal long turn $t$ with branch length $L$ there
  exists a train track morphism $g_t:\Gamma\to\Gamma$ such that $t$ is
  $g_t$-long and mapped by $g_t^{LT}$ to a legal long turn.
\item\label{item:hyp-positive} There exists a train track morphism $h:\Gamma\to\Gamma$ which
  satisfies $|h(e)| \geq 2$ for any edge $e$ of $\Gamma$.
\item\label{item:hyp-hom-eq} All the above maps $g_t$ and $h$ are homotopy equivalences.
\end{enumerate}
Then there exists a legalizing
train track morphism $g:\Gamma \to \Gamma$ which is obtained as a composition of the
$g_t$ and $h$.
\qed
\end{prop}

Before going back to the situation considered in the previous
sections, we first quote the main result of~\cite{CL}, in a slightly
strengthened version due to 
Remark~6.6 of~\cite{CL} and adapted to the
circumstances here:

\begin{thm}[{\cite[Theorem~1.1 and 
Remark~6.6]{CL}}]
\label{thm:quote-2}
Let $\Gamma$ be a graph with gate structure $\mathbf G$, let $f:
\Gamma \to \Gamma$ be a train track morphism with positive transition
matrix $M(f)$ and gate-Whitehead-graph at every vertex that is
connected.  Let $g: \Gamma \to \Gamma$ be a legalizing train track
morphism with respect to the gate structure $\mathbf G$ which is a
homotopy equivalence that fixes every vertex of $\Gamma$ and every
gate of $\mathbf G$.  Then:
\begin{enumerate}
\item\label{item:thm-quote-2-iwip} The map $f \circ g: \Gamma
  \to\Gamma$ is a train track representative of an iwip automorphism
  $\phi \in \Out(\FN)$.
\item\label{item:thm-quote-2-no-inp} There is no periodic INP in
  $\Gamma$ for the train track map $f \circ g$.  In particular there
  are no non-trivial $(f \circ g)$-periodic conjugacy classes in
  $\FN$.
\item\label{item:thm-quote-2-index} The stable index list for $\phi$
  is given by the gate index list for 
$\Gamma$ defined by $\mathbf G$ at the $f$-periodic vertices of $\Gamma$.
\end{enumerate}
\end{thm}

We will now go back to the graph $\Gamma$ from the previous sections,
i.e. with gate structure $\mathbf G$ that realizes the given index
list from Theorem~\ref{thm:main} as gate indices.  We will show below
how to use the previously derived train track morphisms $h$ and $g_t$
on $\Gamma$ via the above quoted results from~\cite{CL} to finish the
proof of Theorem~\ref{thm:main}.

We first observe:

\begin{cor}
\label{cor:g-legalizing}
Let $\Gamma$ be the graph with gate structure $\mathbf G$ defined in
Section~\ref{sec:graph} for the given list of gate indices.  Then
there exists a legalizing train track morphism $g:\Gamma\to\Gamma$
which is a homotopy equivalence and fixes each vertex and each gate of
$\Gamma$.
\end{cor}

\begin{proof} 
We use Proposition~\ref{prop:local-legalizing} to obtain the
hypothesis~(\ref{item:hyp-legalizing}) of
Proposition~\ref{prop:quote-1}, where we note that if $g_t$ legalizes
a long turn $t'$ with branch length $C' \geq 1$, then it also
legalizes any long turn $t$ with branch length $C \geq C'$ which
contains $t'$ as ``subturn''.

We note that hypothesis~(\ref{item:hyp-positive}) of
Proposition~\ref{prop:quote-1} is satisfied by the train track
morphism $h:\Gamma\to\Gamma$ from Proposition~\ref{prop:h-positive}.

Hypothesis~(\ref{item:hyp-hom-eq}) is satisfied, as has been shown in
Propositions~\ref{prop:h-positive} and \ref{prop:local-legalizing}.

Thus Proposition~\ref{prop:quote-1} assures us the existence of the
legalizing map $g$ as product of $h$ and the $g_t$.  Since by
Propositions~\ref{prop:h-positive} and \ref{prop:local-legalizing} the
latter are all homotopy equivalences that fix every vertex of $\Gamma$
and every gate of $\mathbf G$, this proves the claim.
\end{proof}

\begin{proof}[Proof of Theorem~\ref{thm:main}]
We note that the map $h: \Gamma \to \Gamma$ from
Section~\ref{sec:map-iwip-but-inp} satisfies by
Proposition~\ref{prop:h-positive} all of the requirements of the map
$f$ in Theorem~\ref{thm:quote-2}. By Corollary~\ref{cor:g-legalizing}
the same is true for the map $g$ obtained in
Corollary~\ref{cor:g-legalizing}.

Hence part~(\ref{item:thm-quote-2-iwip}) of Theorem~\ref{thm:quote-2}
shows that $h \circ g$ represents an iwip automorphism $\phi$ of $\FN
= \pi_1 \Gamma$, and part~(\ref{item:thm-quote-2-index}) assures that
the stable index list of $\phi$ is equal to the gate index list of
$\Gamma$ with respect to $\mathbf G$, which is built in
Section~\ref{sec:graph} to realize the given list of values $j_1,
j_2,\ldots,j_\ell$, see
Proposition~\ref{prop:graph-index-realization}.
\end{proof}

\begin{rem}
\label{rem:intrinsic-disappeared}
There is a subtlety in the above proof which we would like to point
out to the reader, concerning the topic ``given gate structure''
versus ``intrinsic gate structure'' (as defined in
Section~\ref{sec:prelim-CLP}).  It shows up in relation to two aspects
which are relevant in our context:
\begin{enumerate}
\item[(a)] The index of the automorphism represented by a train track
  morphisms $f: \Gamma \to \Gamma$ with respect to a gate structure
  $\mathbf G$ depends on the intrinsic gate structure $\mathbf G(f)$,
  which may well be strictly finer than $\mathbf G$.
\item[(b)] A map $g: \Gamma \to \Gamma$, which is a train track
  morphism with respect to two distinct gate structures $\mathbf G$
  and $\mathbf G'$, may well be legalizing with respect to $\mathbf G$
  but not with respect $\mathbf G'$.  (In this case, however, $\mathbf
  G$ must be strictly finer than $\mathbf G'$).
\end{enumerate}
In the situation considered above, both potential problems are
resolved as follows by the application of Theorem~\ref{thm:quote-2}:

The train track morphism $h$ constructed in
Section~\ref{sec:map-iwip-but-inp} may indeed well have an intrinsic
gate structure $\mathbf G(h)$ that is strictly finer than the
previously defined gate structure $\mathbf G$.  However, in Lemma~5.9
of~\cite{CL} it has been shown that a train track morphisms $g$ with
respect to a gate structure $\mathbf G$ which is legalizing (with
respect to $\mathbf G$) satisfies indeed $\mathbf G(g) = \mathbf
G$. Now, since the composition of any train track morphism with a
legalizing train track morphism is again legalizing
(see~\cite[Proposition~5.8]{CL}), by the same argument as before we
obtain automatically $\mathbf G(h \circ g) = \mathbf G$.
\end{rem}

\section{Discussion}
\label{sec:proof-discussion}

We will now discuss some further aspects of the index of free group automorphisms:

\subsection{The index deficit}

Handel and Mosher~\cite[\S1.5, Question~5]{HM-axes} ask what values
for the {\em index deficit} $N- \ind_{stab} \phi \, -1$ for any (iwip)
automorphism $\phi \in \Out(\FN)$ are possible, and whether for $N \to
\infty$ the maximal index deficit goes to $\infty$.

From our Theorem~\ref{thm:main} it follows directly that for every $N
\geq 3$ every value $j$ in $\frac{1}{2}\N$ with $0 \leq j \leq N -
\frac{3}{2}$ is achieved as index deficit for some iwip $\phi$.  The
maximal index deficit is hence equal to $N - \frac{3}{2}$, which
indeed tends to $\infty$ for $N \to \infty$. This result has
independently been obtained also by Sonya Leibman~\cite{leibman}.

\subsection{The index of geometric iwips}

We now consider in some detail the results of Masur and
Smillie~\cite{MS}, in particular their Theorem~2 on p.~291: The
translation of the terminology used there for quadratic differentials
and pseudo-Anosov homeomorphisms to the usual terminology for free
group automorphisms is not completely evident. We give here a bit of
translation help:

In the absence of punctures on the surface $M$, a $p_k$-pronged
singularity in~\cite{MS} translates into an isogredience class of
automorphisms $\Phi_k$ with $\ind \Phi_k = j_k = \frac{p_k -
  2}{2}$. In this case we would have to translate the genus $g$ of
$M$, multiplied by 2, into the rank $N$ of the free group, except that
without punctures $\pi_1 M$ is not free, which explains the summand
$-4$ in the index equality in part (a) of Theorem~2 of \cite{MS}.

Now, the $n$ punctures which Masur and Smillie admit in their
Theorem~2 appear nowhere explicitly, but in fact they can be
essentially anywhere: If a puncture lies outside of the singularities
and outside the separatrices raying out from them, then it lies on
some regular leaf of the stable foliation, and hence it becomes a
``$2$-prong singularity'', thus contributing a value $p_k = 2$ to the
given list.  The automorphisms $\Phi_k$ in the corresponding
isogredience class have two attracting fixed points on $\partial \FN$
and ${\rm rank(Fix} \,\,\Phi_k) = 1$, which adds up to $\ind \Phi_k =
1.$

If a puncture coincides with a singularity, say with $p_k$ prongs,
then, by the analogous reasoning, we obtain $\ind \Phi_k =
\frac{p_k}{2}$. This explains also why a value $p_k = 1$, which they
admit, does not lead to negative index of the corresponding
isogredience class: any singularity with a single prong only must
coincide with one of the punctures!

However, in the context of the paper here we have to add a further
restriction: A pseudo-Anosov surface homeomorphism induces an iwip
automorphism if and only if the surface has only one puncture, so that
we have in our context always the condition $n = 1$.

We come now to the 4 exceptional cases listed in part~(c) of their
Theorem~2: The first case $(6; -1)$ can not be realized by a
pseudo-Anosov map with non-orientable stable foliation, but according
to their Theorem~2 there must be a realization by a pseudo-Anosov with
orientable stable lamination.  The last case, $(\, \, \, ; -1)$
requires more than one puncture, so that it is ruled out by the
previous paragraph.  The third case, $(1, 3; -1)$, adds up to $g = 1$,
so that for $n =1$ one has $N = 2$: In this case all automorphisms are
geometric, and hence there is no chance to realize the corresponding
index list $[\frac{1}{2}, \frac{1}{2}]$.  However, in the paper here
we always assume $N \geq 3$, so that this case does not occur.

There 
remains the second exceptional case: $(3, 5; -1)$.  In this case we
have $g = 2$. From $n = 1$ we deduce the following two possibilities
for the index list: $[\frac{5}{2}, \frac{1}{2}]$ or $[\frac{3}{2},
  \frac{3}{2}]$, according to which of the two singularities coincides
with the puncture. The former index list is alternatively realized by
the case $(1, 7; -1)$ from~\cite{MS} (with the puncture at the
singularity with only 1 prong), which satisfies all conditions of
Theorem~2 of~\cite{MS}. The other index list, $[\frac{3}{2},
  \frac{3}{2}]$, however, leads always back to their second
exceptional case $(3, 5; -1)$, and hence can not be realized by a
geometric automorphism. 
There remains as a 
last ``left-over mystery'' of
the index realization problem for $N \geq 3$ the question whether the
index list $[\frac{3}{2}, \frac{3}{2}]$ can perhaps be realized by a
parageometric automorphism of $F_4$.

\subsection{Some numerical experiments regarding the stable index}

Our realization result naturally leads to the question of frequency of
the different index lists. Thanks to the program developed by the
first author in Python and Sage we were able to do the following
numerical experiments.

Fix a finite alphabet $A$ with $N$ letters and let $F_A$ be the free
group on $A$.  With Convention~\ref{convention:no-untouched-edges}, an
elementary Nielsen automorphism is given by $a\mapsto ab$ with $a,b\in
A^{\pm 1}$, $a\neq b$ and $a\neq b^{-1}$. Recall that elementary
Nielsen automorphisms form a generating set of $\Aut(F_A)$ and
$\Out(F_A)$.

Our program tries random products of $L$ elementary Nielsen
automorphisms, that is to say they approximate the random walk on
$\Out(F_A)$ for this generating set. 

Each line of the table 
below 
corresponds to a sample size of $100$ computed
random automorphisms.  Computations were made at the math department
in Marseille, without compiling the Sage code nor looking for serious
optimization.

Note that those automorphisms commonly involve words with several
thousands of letters. Note also that those frequencies are 
not 
completly significant. First, Rivin~\cite{rivin} (see also
\cite{sisto}) proved that the frequency of iwips goes to $100\%$ when
the number of elementary Nielsen automorphisms in the product goes to
infinity (but $26$ or even $41$ are quite small compare to
infinity). Different tests with the above data may lead to slightly
different results. However, what seems to be significant is that:
\begin{enumerate}
\item Automorphisms with small indices are much more frequent that
automorphisms with high indices.
\item Automorphisms with index greater than half the theoretical
  maximal ($\frac{N-1}{2}$) almost never occur. In particular the
  maximal index $N-1$ never occured in our tests out of thousands of
  tries.
\item Several index lists seem to share positive frequency.
\item The smallest index list: $[\frac 12]$, is not always the most frequent.
\end{enumerate}
We have no clue on how to prove or disprove such experimental
observations.

\bigskip
\noindent
\begin{center}
\begin{tabular}{|c|c|c|c|c|c|c|c|c|}
\hline
&&\multicolumn{6}{c|}{Frequency}&\multirow{3}{*}{\parbox{2cm}{\begin{center}Computation\\ time\end{center}}
}\\
\cline{3-8}
N&L&\multirow{2}{*}{iwips}&\multicolumn{5}{c|}{most frequent index lists among iwips}&\\
\cline{4-8}
&&&$[\frac 12]$&$[\frac 12,\frac 12]$&$[1]$&$[\frac 12,\frac 12,\frac 12]$&$[\frac 12,\frac 12,\frac 12,\frac 12]$&\\
\hline
3&26&100\%&64\%&34\%&1\%&0\%&0\%&2 min\\
\hline
4&26&97\%&47\%&34\%&3\%&14\%&1\%&4 min\\
\hline
5&26&93\%&29\%&32\%&3\%&28\%&5\%&7 min\\
\hline
6&29&95\%&21\%&29\%&6\%&20\%&9\%&15 min\\
\hline
7&34&91\%&17\%&26\%&9\%&25\%&7\%&23 min\\
\hline
8&36&84\%&13\%&19\%&7\%&17\%&18\%&31 min\\
\hline
9&39&78\%&11\%&6\%&11\%&18\%&10\%&46 min\\
\hline
10&41&76\%&3\%&8\%&8\%&5\%&8\%&1h5min\\
\hline
\end{tabular}
\end{center}

\bigskip

\noindent Institut de mathématiques de Marseille\\ 
Université d'Aix-Marseille\\
39, rue Frédéric Joliot Curie \\
13453 Marseille Cedex 13\\
France\\
\href{mailto:thierry.coulbois@univ-amu.fr}{\nolinkurl{thierry.coulbois@univ-amu.fr}}\\
\href{mailto:martin.lustig@univ-amu.fr}{\nolinkurl{martin.lustig@univ-amu.fr}}

\end{document}